\documentclass{article}
\usepackage{epsfig}
\usepackage{amsmath}
\usepackage{amsfonts}
\usepackage{amsthm}
\usepackage{algorithm}
\usepackage{algorithmic}
\usepackage{graphicx}
\usepackage{subfigure}

\newcommand {\C}        {{\mathbb{C}}}

 \newtheorem{theorem}{Theorem}[section]

 \newtheorem{corollary}[theorem]{Corollary}

\title{ Matrix Polynomials with Specified Eigenvalues }
\author{
Michael~Karow\thanks{ Department of Mathematics, TU-Berlin {\tt (karow@math.tu-berlin.edu.tr)}. }
\and
Emre~Mengi\thanks{ Department of Mathematics, Ko\c{c} University,
Rumelifeneri Yolu, 34450 Sar{\i}yer-\.{I}stanbul, Turkey {\tt (emengi@ku.edu.tr)}. 
The work of this author was supported in part by the European Commision grant
PIRG-GA-268355 and the T\"{U}B\.{I}TAK (the scientific and technological research council
of Turkey) carrier grant 109T660.
 }
}

\begin{document}
\maketitle

\begin{abstract}
\noindent
This work concerns the distance in 2-norm from a matrix polynomial to a nearest polynomial 
with a specified number of its eigenvalues at specified locations in the complex plane.
Perturbations are allowed only on the constant coefficient matrix. Singular value optimization 
formulas are derived for these distances facilitating their computation. The singular value 
optimization problems, when the number of specified eigenvalues is small, can be solved 
numerically by exploiting the Lipschitzness and piece-wise analyticity of the singular values 
with respect to the parameters. \\


\end{abstract}

\pagestyle{myheadings}
\thispagestyle{plain}
\markboth{ M. KAROW AND E. MENGI }{ POLYNOMIALS WITH SPECIFIED EIGENVALUES }

\section{Introduction}

We study the distance from a matrix polynomial to a nearest polynomial with specified number
of eigenvalues at specified positions in the complex plane. Formally, let $P : {\mathbb C} \rightarrow {\mathbb C}^{n\times n}$,
defined by
\begin{equation}\label{eq:matrix_poly}
	P(\lambda) := \sum_{j=0}^m  \lambda^j A_j,
\end{equation}
be a square matrix polynomial where $A_j \in \C^{n\times n}$. Throughout the paper,
we will assume that ${\rm rank}(A_m) = n$. Suppose also that a set 
${\mathbb S}	:=	\{	\lambda_1, \dots, \lambda_s 	\}$ consisting of complex scalars and 
a positive integer $r$ are given. This paper provides a singular value formula for the distance 
\begin{equation}\label{eq:distance_defn}
	\tau_r({\mathbb S})
			:=
		\inf
		\left\{
				\| \Delta \|_2	\;\;\;	|	\;\;\;	
					\Delta \in {\mathbb C}^{n\times n}
						\;\;\; 	{\rm s.t.}	\;\;\;
							\sum_{j=1}^s
							m_j	
							\left( 
								P + \Delta 		
							\right)
									\geq 
								r
		\right\}
\end{equation}
where $m_j(P + \Delta)$ denotes the algebraic multiplicity of $\lambda_j$ as an eigenvalue of 
$P_\Delta(\lambda) := P(\lambda) + \Delta$, that is the multiplicity of $\lambda_j$ as a root of the 
polynomial $\det \left( P_\Delta(\lambda) \right)$.

The formula derived is a generalization of the singular value characterization in \cite{Kressner2011}
for a linear matrix pencil of the form $L(\lambda) = A_0 + \lambda A_1$, which was inspired by
the Malyshev's work \cite{Malyshev1999} earlier. However, unlike \cite{Malyshev1999} the derivation
here fully depends on a Sylvester equation characterization for a matrix polynomial to have
sufficiently many eigenvalues belonging to ${\mathcal S}$. This yields a neater derivation.
The distance from a matrix polynomial to a nearest one with a multiple eigenvalue was considered
in \cite{Papat2008}, where singular value formulas yielding lower and upper bounds were derived.
In \cite{Papat2008} perturbations to all of the coefficient matrices were allowed, but it is not
clear how tight the derived bounds are. Here the derived singular value formula, when the number of 
prescribed eigenvalues in ${\mathcal S}$ is small, facilitates the numerical computation of the distances by 
means of the algorithms exploiting the Lipschitzness \cite{Piyavskii1972, Shubert1972} and piece-wise 
analyticity of singular values \cite{Kilic2012}.

The Sylvester equation that we utilize is of the form
\[
	A_0 X	+	A_1 X C	+ 	A_2 X C^2	+	\dots		+	A_m X C^m = 0.
\]
Our approach is based on seeking an upper triangular $C$ so that the linear space consisting of 
matrices $X$ satisfying this equation is of dimension at least $r$. A pair $(X, C)$ satisfying the
equation is named as an invariant pair in \cite{Betcke2011}, where a perturbation theory and
numerical approaches are developed for an invariant pair. In the special case when $C$ is
in the Jordan canonical form, the pair $(X,C)$ is called a Jordan pair \cite[Chapter 2]{Gohberg1982}. 
In the extreme case when all eigenvalues are prescribed so that $r = mn$, then the pair $(X,C)$
is closely related to a (right) standard paper \cite[Chapter 2]{Gohberg1982}, \cite[Chapter 5]{Gohberg2006},
which has an important place for linearizations of matrix polynomials.

In the next section, we derive the characterization in terms of the Sylvester equation above  
for the condition $\sum_{j=1}^s   m_j ( P )   \geq   r $. Then we turn the Sylvester characterization
into a rank problem. The rank characterization provides a singular value formula bounding 
the actual distance from below right away due to the Eckart-Young theorem. In Section \ref{sec:optimal_pert} we 
establish the exact equality of the singular value formula with the distance by constructing an optimal 
perturbation. The derived singular value formula (Theorem \ref{thm:main_result}) can be conveniently 
expressed in terms of divided differences (Theorem \ref{thm:main_result2} and Corollary \ref{thm:main_result3}). 
Section \ref{sec:num_examp} illustrates the validity of the results in practice on two examples and making 
connections with the $\epsilon$-pseudospectrum for a matrix polynomial.

\section{Rank characterization for polynomials with specified eigenvalues}

We first deduce a rank characterization, for a given set of complex scalars
	${\mathbb S}	:=	\{	\lambda_1, \dots, \lambda_s 	\}$
and a positive integer $r$, that confirms whether the scalars in ${\mathbb S}$ are eigenvalues 
of the polynomial $P(\lambda)$ as defined in (\ref{eq:matrix_poly}) with algebraic multiplicities summing 
up to $r$ or greater. Formally, we are seeking a rank characterization for the condition
 \[
 	\sum_{j=1}^s	m_j(P)	\geq 		r.
 \]
 
 The derivation exploits the companion form linearization 
 ${\mathcal L}(\lambda) := {\mathcal A}  +  \lambda {\mathcal B}$ for $P(\lambda)$ with
 \begin{equation}\label{eq:linearization}
 	{\mathcal A}
				:=
		\left[
			\begin{array}{cccc}
					0 	& I		&		& 0	\\
						&		& \ddots    &  	\\
					0	& 0		&	&	I \\
					A_0	& A_1	&	& A_{m-1} \\
			\end{array}
		\right]	\;\; {\rm and} \;\;
	{\mathcal B}
				:=
		\left[
			\begin{array}{cccc}
					-I &		& 0	& 0	 \\
					  & \ddots	& 	&	  \\
					0 &		& -I	& 0	  \\
					0 &		& 0	& A_m \\
			\end{array}
		\right],
 \end{equation}
and benefits from the fact that the eigenvalues of ${\mathcal L}(\lambda)$ and $P(\lambda)$ are 
the same with the same algebraic multiplicities. Due to the assumption that ${\rm rank}(A_m) = n$
the matrix ${\mathcal B}$ is full rank. Consequently, we could apply Theorem \ref{thm:Syl_Pencil} 
concerning the multiplicities of the eigenvalues of matrix pencils given below to the pencil ${\mathcal L}(\lambda)$. 
The result  originally appeared in \cite[Theorem 3.3]{Kressner2011} in a more general setting. For the 
theorem we introduce the notation
\begin{equation} \label{eq:cmu}
 	C(\mu,\Gamma)
		=
		\left[
			\begin{array}{cccc}
				\mu_1 &  \gamma_{21} & \dots &  \gamma_{r1} \\
					0 &	\mu_2	     & \ddots & \vdots \\[-0.1cm]
					   &			     & \ddots & \gamma_{r(r-1)}		\\
					 0 &			     &		  & \mu_r \\		
			\end{array}
		\right],
\end{equation}
where
\[
	  \mu 
			=
		\left[
			\begin{array}{cccc}
				\mu_1 	&	\mu_2	& 	\dots 	&	\mu_r
			\end{array}
		\right]^T \in {\mathbb S}^r \quad {\rm and} \quad
	  \Gamma
			=
		\left[
			\begin{array}{cccc}
				\gamma_{21} & \gamma_{31} & \dots & \gamma_{r,r-1}
			\end{array}
		\right]^T \in \C^{r(r-1)/2}	
\]
and ${\mathbb S}^r$ represents the $r$ tuples with elements from the set ${\mathbb S}$.
We also denote the generic set of $\Gamma$ values such that $C(\mu,\Gamma)$ 
has all eigenvalues with geometric multiplicities equal to one by ${\mathcal G}(\mu)$
(for genericity of such $\Gamma$ values see \cite{Demmel1995}).

\begin{theorem}\label{thm:Syl_Pencil}
	Let $L(\lambda) := A + \lambda B$ be a matrix pencil with $A,B \in {\mathbb C}^{n\times n}$ and such that 
	${\rm rank}(B) = n$, ${\mathbb S}	:=	\{	\lambda_1, \dots, \lambda_s 	\}$ be a set of complex scalars, 
	and $r\in {\mathbb Z}^+$. The following two conditions are equivalent:
	\begin{enumerate}
		\item[\bf (1)]	$ \sum_{j=1}^s m_j (A,B)	\geq 		r$ where $m_j(A,B)$ is the algebraic multiplicity of
		$\lambda_j$ as an eigenvalue $L(\lambda) = A+\lambda B$.
		\item[\bf (2)]	There exists a $\mu \in {\mathbb S}^r$ such that for all $\Gamma \in {\mathcal G}(\mu)$
				\[
					{\rm dim}
						\left\{
							X\in {\mathbb C}^{n\times r}	\;
										|	\;
							AX	+	BXC(\mu,\Gamma)	=	0
						\right\}	\geq r.
				\]
	\end{enumerate}
\end{theorem}

\begin{theorem}\label{thm:Syl_Polynomial}
	Let $P(\lambda) := \sum_{j=0}^m \lambda^j A_j$ with $A_j \in {\mathbb C}^{n\times n}$ and such that 
	${\rm rank}(A_m) = n$, ${\mathbb S}	:=	\{	\lambda_1, \dots, \lambda_s 	\}$ be a set of complex 
	scalars, and $r\in {\mathbb Z}^+$. The following two conditions are equivalent:
	\begin{enumerate}
		\item[\bf (1)]	$ \sum_{j=1}^s m_j (P)	\geq 		r$ where $m_j(P)$ is the algebraic multiplicity of
		$\lambda_j$ as an eigenvalue $P(\lambda)$.
		\item[\bf (2)]	There exists a $\mu \in {\mathbb S}^r$ such that for all $\Gamma \in {\mathcal G}(\mu)$
				\[
					{\rm dim}
						\left\{
							X\in {\mathbb C}^{n\times r}	\;
										|	\;
							\sum_{j=0}^m 	A_jXC^j(\mu,\Gamma)	=	0
						\right\}	\geq r.
				\]
	\end{enumerate}
\end{theorem}
\begin{proof}
We apply Theorem \ref{thm:Syl_Pencil} to the linearization (\ref{eq:linearization}) for $P(\mu)$. It follows
from Theorem \ref{thm:Syl_Pencil} that the condition $ \sum_{j=1}^s m_j (P) \geq r$ is met if and only if
\[
			{\rm dim}
				\{
					{\mathcal X}\in {\mathbb C}^{mn\times r}	\;
						|	\;
					{\mathcal A}{\mathcal X}	+	{\mathcal B}{\mathcal X}C(\mu,\Gamma)	=	0
				\}	\geq r.
\]
On the other hand the partitioning 
	${\mathcal X}
		=
		\left[
			\begin{array}{cccc}
				X_0^T	&	X_1^T	&	\dots		&	X_{m-1}	
			\end{array}
		\right]^T
	$ 
where $X_j \in {\mathbb C}^{n\times r}$ reveals that the condition
\[
	0 = {\mathcal A}{\mathcal X}	+	{\mathcal B}{\mathcal X}C(\mu,\Gamma)
	   =
	   	\left[
				\begin{array}{c}
					X_1 \\
					\vdots \\
					X_{m-1} \\
					\sum_{j=0}^{m-1} A_j X_j
				\end{array}
		\right]
				+
		\left[
				\begin{array}{c}
					-X_0 C(\mu,\Gamma) \\
					\vdots \\
					-X_{m-2} C(\mu,\Gamma) \\
					A_m X_{m-1} C(\mu,\Gamma) \\
				\end{array}
		\right]
\]
could be expressed as $X_j = X_{j-1} C(\mu,\Gamma)$ for $j = 1,\dots,m-1$ and
\[
	\sum_{j=0}^{m-1} A_j X_j		+	A_m X_{m-1} C(\mu,\Gamma)	=	0.
\]
By eliminating $X_j$ for $j=0,\dots,m-1$ in the last equation using $X_j = X_0 C(\mu,\Gamma)^j$ 
we obtain
\[
	\sum_{j=0}^{m} A_j X_0 C(\mu,\Gamma)^j	  =		0.
\]
To summarize $X_0$ is a solution of
$
	\sum_{j=0}^{m} A_j X C(\mu,\Gamma)^j	  =		0	
$
if and only if
\[
	{\mathcal X}_0
		=
		\left[
			\begin{array}{cccc}
				X_0^T	&	\left( X_0 C(\mu,\Gamma)\right)^T	&	\dots		&	\left( X_{0} C^{m-1}(\mu,\Gamma) \right)^T	
			\end{array}
		\right]^T
\]
is a solution of
$
	{\mathcal A}{\mathcal X}	+	{\mathcal B}{\mathcal X}C(\mu,\Gamma)	=	0
$
and the result follows.
\end{proof}

As discussed in the introduction a pair $(X, C(\mu,\Gamma))$ satisfying the Sylvester equation
\[
		\sum_{j=0}^m 	A_jXC^j(\mu,\Gamma)	=	0
\]
is called an invariant pair of the matrix polynomial $P(\lambda)$ \cite{Betcke2011}.
When $C(\mu,\Gamma)$ is diagonal, it can trivially be verified that $P(\mu_j) x_j = 0$ 
for $j = 1,\dots, r$ where $x_j$ denotes the $j$th column of $X$, so the columns of $X$ are 
eigenvectors of $P$. Another special case is a Jordan pair when $C(\mu,\Gamma)$ is in the Jordan form.
In this case, it can be shown that the columns of $X$ are Jordan chains of $P$ \cite{Gohberg1982}.
Thus, a matrix $X$ satisfying the Sylvester equation above is inherently related to the generalized
eigenspaces of $P$. The dimension of all such $X$ is related to the dimensions of the generalized
eigenspaces as revealed by Theorem \ref{thm:Syl_Polynomial}. Next we express the Sylvester characterization 
in Theorem \ref{thm:Syl_Polynomial} as a rank condition involving matrices in terms of the
Kronecker product $\otimes$.

\begin{corollary}\label{cor:rank_Polynomial}
	Let $P(\lambda) := \sum_{j=0}^m \lambda^j A_j$ with $A_j \in {\mathbb C}^{n\times n}$ and such that
	${\rm rank}(A_m) = n$, ${\mathbb S}	:=	\{	\lambda_1, \dots, \lambda_s 	\}$ be a set of complex 
	scalars, and $r\in {\mathbb Z}^+$. The following two conditions are equivalent.
	\begin{enumerate}
		\item[\bf (1)]	$ \sum_{j=1}^s m_j (P)	\geq 		r$ where $m_j(P)$ is the algebraic multiplicity of
		$\lambda_j$ as an eigenvalue $P(\lambda)$.
		\item[\bf (2)]	There exists a $\mu \in {\mathbb S}^r$ such that for all $\Gamma \in {\mathcal G}(\mu)$
				\[
					{\rm rank}
					\left(
							\sum_{j=0}^m 	\left( C^j(\mu,\Gamma) \right)^T \otimes A_j
					\right)	\leq n\cdot r -r.
				\]
	\end{enumerate}
\end{corollary}
\begin{proof}
By reserving the notation ${\rm vec}(\cdot)$ for the linear operator that stacks up the columns 
of its matrix argument into a vector, the result follows from Theorem \ref{thm:Syl_Polynomial}
and the identity
\[
	{\rm vec} (AXB) =  \left( B^T \otimes A \right) {\rm vec}(X),
\]
specifically from an application of the identity above to
$
	\sum_{j=0}^m 	A_jXC^j(\mu,\Gamma)	=	0.
$
\end{proof}

For instance we deduce the following when ${\mathcal S} = \{ \mu \}$ and $r=2$ from
the corollary above; the matrix polynomial $P(\lambda)$ has $\mu$ as a multiple eigenvalue 
if and only if
\[
	{\rm rank}
	\left(
			\sum_{j=0}^m
				\left[
					\begin{array}{cc}	
						\mu	&	0 \\
						\gamma	 &	\mu \\
					\end{array}
				\right]^j
					\otimes 
			A_j
	\right)
			=
	{\rm rank}
	\left(
		\left[
			\begin{array}{cc}
				P(\mu)	&	0	\\
				\gamma P'(\mu)	&	P(\mu)	\\
			\end{array}
		\right]
	\right)	\leq 2n - 2	
\]
for all $\gamma \neq 0$.

\section{ Derivation of the Singular Value Formula }\label{sec:optimal_pert}

For each $\mu \in {\mathbb S}^r$ let us define the quantity
\[
	{\mathcal P}_r(\mu)	:=	
				\inf
				\left\{
						\| \Delta  \|_2	\;\;  |  \;\;
							{\rm rank}
								\left(
									{\mathcal Q}(\mu,\Gamma,P+\Delta)
								\right)  \leq n\cdot r -r
				\right\}
\]
for any $\Gamma \in {\mathcal G}(\mu)$, where
\begin{equation}\label{eq:matrix_fun}
	{\mathcal Q}(\mu,\Gamma,P) := \sum_{j=0}^m 	\left( C^j(\mu,\Gamma) \right)^T \otimes A_j,
\end{equation}
and $P_\Delta := P + \Delta$ denotes the polynomial $P_\Delta(\lambda) := P(\lambda) + \Delta$.
Then, from Corollary \ref{cor:rank_Polynomial}, the distance to a nearest polynomial with 
specified eigenvalues could be expressed as
\[
		\tau_r({\mathbb S})
			:=
		\inf_{\mu \in {\mathbb S}^r}
				{\mathcal P}_r(\mu),
\]
so it suffices to derive a singular value formula for ${\mathcal P}_r(\mu)$.

We immediately deduce the lower bound
\begin{equation}\label{eq:defn_sval}
	{\mathcal P}_r(\mu)
			\geq
	\sup_{\Gamma \in {\mathbb C}^{r(r-1)/2}} \;
	\sigma_{-r} 
		\left(
			{\mathcal Q}(\mu,\Gamma,P)
		\right)
			:=
	\kappa_r(\mu),
\end{equation}
since for any matrix $B$ the distance in 2-norm to a nearest matrix of rank $\ell$
is given by $\sigma_{\ell+1}(B)$ by the Eckart-Young theorem. Here and elsewhere
$\sigma_{-k}(\cdot)$ denotes the $k$th smallest singular value of its matrix argument.
Note that, when deducing the lower bound in (\ref{eq:defn_sval}), we also benefit from 
the continuity of 
		$\sigma_{-r} 
		\left(
			{\mathcal Q}(\mu,\Gamma,P)
		\right)
		$
with respect to $\Gamma$, as well as the genericity of the set ${\mathcal G}(\mu)$,
so that the supremum is over all $\Gamma \in {\mathbb C}^{r(r-1)/2}$ rather than
$\Gamma \in {\mathcal G}(\mu)$. We could not immediately 
deduce the upper bound because the allowable perturbations have special structure, 
i.e., they are of the form $I \otimes \Delta$.

To establish the validity of the reverse inequality ${\mathcal P}_r(\mu) \leq \kappa_r(\mu)$
it is sufficient to construct a perturbation $\Delta_{\ast}$ such that
\begin{enumerate}
	\item[\bf (i)]	$\| \Delta_{\ast} \|_2 = \kappa_r(\mu)$, and
	\item[\bf (ii)]	${\rm rank}
					\left(
									{\mathcal Q}(\mu,\Gamma,P + \Delta_\ast)
					\right)  \leq n\cdot r -r $
				for some $\Gamma\in {\mathcal G}(\mu)$.
\end{enumerate}
As shown in the appendix the supremum in (\ref{eq:defn_sval}) is attained for generic $\mu$. 
For such a generic $\mu$, let $\Gamma_\ast$ be a point where this supremum is attained, that is
\begin{equation}\label{eq:defn_gammaast}
	\kappa_r(\mu)
				=
	\sigma_{-r} 
		\left(
			{\mathcal Q}(\mu,\Gamma_\ast,P)
		\right).
\end{equation}
Let $U, V \in {\mathbb C}^{nr}$ be a consistent pair of unit left and right singular 
vectors associated with this singular value, in particular $U$ and $V$ 
satisfy
\begin{equation}\label{eq:opt_sval1}
		{\mathcal Q}(\mu,\Gamma_\ast,P) \cdot	
			V
						=
		\kappa_r(\mu)	\cdot U
\end{equation}
and
\begin{equation}\label{eq:opt_sval2}
		U^\ast \cdot
		{\mathcal Q}(\mu,\Gamma_\ast,P)
						=
		\kappa_r(\mu) \cdot V^\ast.
\end{equation}
In the subsequent two subsections we prove that 
\begin{equation}\label{eq:optimal_perturb}
	\Delta_\ast
			:=
		-\kappa(P,\mu) {\mathcal U} {\mathcal V}^+
\end{equation}
satisfies both of the properties \textbf{(i)} and \textbf{(ii)} above, where ${\mathcal U}, {\mathcal V} \in {\mathbb C}^{n\times r}$
are such that $U = {\rm vec}({\mathcal U})$ and $V = {\rm vec}({\mathcal V})$ under the following mild
assumptions.
\begin{enumerate}
	\item \textbf{(Multiplicity Assumption)} The multiplicity of
	$
		\sigma_{- r } 
		\left(
			{\mathcal Q}(\mu,\Gamma_\ast,P)
		\right)	
	$
	is one.
	\item \textbf{(Linear Independence Assumption)} ${\rm rank}({\mathcal V}) = r$
\end{enumerate}

\subsection{ Norm of $\Delta_\ast$ }\label{sec:norm_cond}


We aim to show that $\| \Delta_\ast \|_2 = \kappa_r\left(\mu \right)$. For this
purpose, it is sufficient to establish the validity of ${\mathcal U}^\ast {\mathcal U} = {\mathcal V}^\ast {\mathcal V}$, since this property
implies
\[
	\| {\mathcal U} {\mathcal V}^+ \|_2 = \max_{w \in \C^n, \; \| w \|_2 = 1} \sqrt{ \left( {\mathcal V}^+ \right)^\ast {\mathcal U}^\ast {\mathcal U} {\mathcal V}^+ w}
			   = \max_{w \in \C^n, \; \| w \|_2 = 1} \sqrt{ \left( {\mathcal V}^+ \right)^\ast {\mathcal V}^\ast {\mathcal V} {\mathcal V}^+ w}
			   = \| {\mathcal V} {\mathcal V}^+ \|_2 = 1,
\] 
where the last equality is due to the fact that ${\mathcal V}{\mathcal V}^+$ is an orthogonal projector.

Throughout the rest of this subsection we prove the property ${\mathcal U}^\ast {\mathcal U} = {\mathcal V}^\ast {\mathcal V}$ under the 
multiplicity assumption. Let
\begin{equation*}
\sigma(\Gamma):=\sigma_{ -r }({\mathcal Q}(\mu, \Gamma,P))
\end{equation*}
Then the partial derivatives of ${\mathcal Q}(\mu, \Gamma,P)$ with respect to the real
and the imaginary parts of the components $\gamma_{ik}$ of $\Gamma$ are 
\begin{eqnarray*}
\frac{\partial {\mathcal Q}}{\partial \Re\gamma_{ik}}(\mu,\Gamma,P)
&=&
\sum_{j=1}^m \left(\sum_{\ell=0}^{j-1}
 C^{\ell}(\mu,\Gamma)
\frac{\partial C(\mu,\Gamma)} {\partial \Re\gamma_{ik}}
C^{j-1-\ell}(\mu,\Gamma)\right)^T
 \otimes A_j
\\
&=&
\sum_{j=1}^m \sum_{\ell=0}^{j-1}
\left( C^{\ell}(\mu,\Gamma)
e_ie_k^\top 
C^{j-1-\ell}(\mu,\Gamma)\right)^T
 \otimes A_j,
\\
\frac{\partial {\mathcal Q}}{\partial \Im\gamma_{ik}}(\mu,\Gamma,P)
&=&
\sum_{j=1}^m \left(\sum_{\ell=0}^{j-1}
C^{\ell}(\mu,\Gamma)
\frac{\partial C(\mu,\Gamma)} {\partial \Im\gamma_{ik}}
C^{j-1-\ell}(\mu,\Gamma)\right)^T
 \otimes A_j\\
&=&
{\imath}\sum_{j=1}^m \sum_{\ell=0}^{j-1}
\left( C^{\ell}(\mu,\Gamma)
e_ie_k^\top 
C^{j-1-\ell}(\mu,\Gamma)\right)^T
 \otimes A_j,
\end{eqnarray*}
where $e_i$ ($e_k$) denotes the $i$th (the $k$th)  column of the
$r\times r$ unit matrix, and $1\leq i<k\leq r$.
Let
\begin{equation}
G:=
\sum_{j=1}^m \sum_{\ell=0}^{j-1}C^{j-1-\ell}(\mu,\Gamma_\ast)
{\mathcal U}^\ast A_j\,{\mathcal V}\,
 C^{\ell}(\mu,\Gamma_\ast).
\end{equation}
 From the assumption that the singular value $\sigma(\Gamma_\ast)$ is 
simple it follows that the function $\Gamma \mapsto \sigma(\Gamma)$
is analytic at $\Gamma_\ast$, and  
\begin{eqnarray*}
0 &=&
\frac{\partial \sigma} {\partial \Re\gamma_{ik}}(\Gamma_\ast)\\
&=&
\Re \left( U^\ast \,
\frac{\partial {\mathcal Q}}{\partial \Re\gamma_{ik}}(\mu,\Gamma_\ast,P)
V
\right)\\
&=&
\Re \left({\rm vec}({\mathcal U})^\ast \,
\frac{\partial {\mathcal Q}}{\partial \Re\gamma_{ik}}(\mu,\Gamma_\ast,P)
{\rm vec}({\mathcal V})
\right)\\
&=&
\Re \left({\rm vec}({\mathcal U})^\ast \,
{\rm vec}
\left(
\sum_{j=1}^m \sum_{\ell=0}^{j-1}
A_j\,{\mathcal V}\,
 C^{\ell}(\mu,\Gamma_\ast)
e_ie_k^\top 
C^{j-1-\ell}(\mu,\Gamma_\ast)
\right)\right)
\\
&=&
\Re\, \left( {\rm tr}\left(
{\mathcal U}^\ast
\sum_{j=1}^m \sum_{\ell=0}^{j-1}
A_j\,{\mathcal V}\,
 C^{\ell}(\mu,\Gamma_\ast)
e_ie_k^\top 
C^{j-1-\ell}(\mu,\Gamma_\ast)
\right) \right)
\\
&=&
\Re\,\big(\,
e_k^\top 
G
e_i\big)\qquad\qquad\text{for } 1\leq i<k\leq r.
\end{eqnarray*}  
The latter equation follows from the trace identity  ${\rm
  tr}(XY)={\rm tr}(YX)$. Analogously we have
\begin{equation*}
0=\frac{\partial \sigma} {\partial \Im\gamma_{ik}}(\Gamma_\ast)
= \Re ({\imath}  e_k^T Ge_i) \; = \; -\Im (e_k^T Ge_i)\qquad \text{for } 1\leq k<i\leq r.
\end{equation*}
Thus, $G$ is upper triangular. Let 
\begin{equation}
M=-{\mathcal U}^\ast A_0{\mathcal V}+\sum_{j=1}^m\sum_{\ell=1}^{j-1} C^{j-\ell}(\mu,\Gamma_\ast){\mathcal U}^\ast A_j{\mathcal V}C^{\ell}(\mu,\Gamma_\ast).
\end{equation}
Then, it is easily verified that
\begin{eqnarray*}
G\, C(\mu,\Gamma_\ast) &=&M+ {\mathcal U}^\ast\sum_{j=0}^m A_j{\mathcal V}C^j(\mu,\Gamma_\ast)\\
&=&M+\sigma(\Gamma_\ast)U^\ast U,
\end{eqnarray*}
where the last equality follows by writing (\ref{eq:opt_sval1}) in matrix form. Also,
\begin{eqnarray*}
C(\mu,\Gamma_\ast)\,G &=&
M+ \left(\sum_{j=0}^mC^j(\mu,\Gamma_\ast){\mathcal U}^\ast A_j\right){\mathcal V}\\
&=&M+\sigma(\Gamma_\ast){\mathcal V}^\ast {\mathcal V},
\end{eqnarray*}
where the last equality follows from (\ref{eq:opt_sval2}). Thus,
\begin{equation}
\sigma(\Gamma_\ast)({\mathcal U}^\ast {\mathcal U}-{\mathcal V}^\ast {\mathcal V})=GC(\mu,\Gamma_\ast)-C(\mu,\Gamma_\ast)G.
\end{equation}
Since $G$ and $C(\mu,\Gamma_\ast)$ are both upper triangular, the right
hand side of this equation is strictly upper triangular. The
left hand side is Hermitian. Hence, both sides vanish. Thus,
${\mathcal U}^\ast {\mathcal U} = {\mathcal V}^\ast {\mathcal V}$.

\subsection{ Sylvester Equation for Perturbed Matrix Polynomial }\label{sec:rank_suff}

In this subsection we show that
\begin{equation}\label{eq:Sylvester}
	{\rm dim}
		\left\{
			X\in {\mathbb C}^{n\times r}	\;\;
				|	\;\;
			\sum_{j=0}^m 	A_jXC^j(\mu,\Gamma_\ast) + \Delta_\ast X		=	0
		\right\}	\geq r
\end{equation}
under the assumption that ${\mathcal V}$ is full rank, where $\Gamma_\ast$ is defined as in (\ref{eq:defn_gammaast}). 
This is equivalent to the satisfaction of the condition
\[
		{\rm rank}
					\left(
							{\mathcal Q}(\mu,\Gamma_\ast,P+\Delta_\ast)
					\right)  \leq n\cdot r -r .
\]
Our starting point is the singular value equation (\ref{eq:opt_sval1}), which could be rewritten 
as a matrix equation of the form
\[
	\sum_{j=0}^m 	A_j {\mathcal V} C^j(\mu,\Gamma_\ast)	=	\kappa_r(\mu) {\mathcal U}.
\]

Assuming ${\mathcal V}$ is full rank we have ${\mathcal V}^+ {\mathcal V} = I$. Consequently,
\[
	\sum_{j=0}^m 	A_j {\mathcal V} C^j(\mu,\Gamma_\ast)	=	\kappa(P,\mu) {\mathcal U} {\mathcal V}^+ {\mathcal V} \;\;		\Longrightarrow	\;\;
	\sum_{j=0}^m 	A_j {\mathcal V} C^j(\mu,\Gamma_\ast) +  \Delta_\ast {\mathcal V}	=	0.
\]
Moreover, consider the subspace of matrices
\[
	{\mathcal D}
		:=
	\{
		D\in {\mathbb C}^{r\times r}	\;\; |  \;\;
		C(\mu,\Gamma_\ast) D - D C(\mu,\Gamma_\ast) = 0
	\}
\]
commuting with $C(\mu,\Gamma_\ast)$, which is of dimension at least $r$ (due to \cite[Theorem 1, p. 219]{Gantmacher1959}). 
For all $D \in {\mathcal D}$, we have
\[
	0 =
	\sum_{j=0}^m 	A_j {\mathcal V} C^j(\mu,\Gamma_\ast)D + \Delta_\ast {\mathcal V}D	=
	\sum_{j=0}^m 	A_j ({\mathcal V}D) C^j(\mu,\Gamma_\ast) +  \Delta_\ast ({\mathcal V}D)
\]
meaning each matrix in the set
$
	\{
		{\mathcal V}D	\;\; |  \;\;
		D \in {\mathcal D}
	\}
$
is a solution of the Sylvester equation 
\[
	\sum_{j=0}^m 	A_jXC^j(\mu,\Gamma_\ast) + \Delta_\ast X		=	0.
\]
Therefore, we conclude with (\ref{eq:Sylvester}) assuming ${\mathcal V}$ is full rank.

\subsection{ Main Result }\label{subsec:main_result}

Let us first suppose $\mu$ consists of distinct scalars. Then all eigenvalues of $C(\mu,\Gamma)$ 
have algebraic and geometric multiplicities equal to one for all $\Gamma$, implying
${\mathcal G}(\mu) = {\mathbb C}^{r(r-1)/2}$. Consequently, we have $\Gamma_\ast \in {\mathcal G}(\mu)$,
where $\Gamma_\ast$ is defined as in (\ref{eq:defn_gammaast}). Furthermore, let us suppose
that $\mu$ takes one of those generic values (specifically $\mu$ satisfies the hypotheses of
Theorem \ref{thm:decay_polynomial}) so that the supremum in (\ref{eq:defn_sval})
is attained. It follows from Sections \ref{sec:norm_cond} and \ref{sec:rank_suff} that ${\mathcal P}_r(\mu) = \kappa_r(\mu)$ 
under multiplicity and linear independence assumptions at the optimal $\Gamma_\ast$.

When there are repeated scalars in $\mu$ or $\mu$ does not take one of the generic values, then there are $\tilde{\mu}$ 
comprised of distinct scalars, that belong to the generic set and arbitrarily close to $\mu$, where the equality 
${\mathcal P}_r(\tilde{\mu}) = \kappa_r(\tilde{\mu})$ is satisfied under multiplicity and linear independence assumptions. 
Then, the equality ${\mathcal P}_r(\mu) = \kappa_r(\mu)$ follows from the continuity of both ${\mathcal P}_r(\cdot)$ and $\kappa_r(\cdot)$ 
with respect to $\mu$ (again under multiplicity and linear independence assumptions). We arrive at the following main 
result of this paper.

\begin{theorem}[\textbf{Distance to Polynomials with Specified Eigenvalues}]\label{thm:main_result}
	Let $P(\lambda) := \sum_{j=0}^m \lambda^j A_j$ with $A_j \in {\mathbb C}^{n\times n}$ and such that
	${\rm rank}(A_m) = n$, ${\mathbb S}	:=	\{	\lambda_1, \dots, \lambda_s 	\}$ be a set of complex 
	scalars, and $r\in {\mathbb Z}^+$. 
	\begin{enumerate}
	\item[\bf (i)]
	Then the singular value characterization
	\begin{equation}\label{eq:main_sval_char}
		\tau_r({\mathbb S})
				=
		\inf_{\mu \in {\mathbb S}^r} \sup_{\Gamma\in {\mathbb C}^{r(r-1)/2}} \;
				\sigma_{- r}
					\left(
						{\mathcal Q}(\mu,\Gamma,P)
					\right)
	\end{equation}
	holds, for the distance $\tau_r({\mathbb S})$ defined as in (\ref{eq:distance_defn}) in terms of
	the matrix function ${\mathcal Q}(\mu,\Gamma,P)$ defined as in (\ref{eq:matrix_fun}), provided that the
	multiplicity and linear independence assumptions hold at the optimal $\mu \in {\mathbb S}^r$
	and $\Gamma \in {\mathbb C}^{r(r-1)/2}$, and if $r > n$ provided that the inner supremum is attained.
	\item[\bf (ii)]
	The minimal $\Delta_\ast$ in 2-norm such that $\sum_{j=1}^s m_j \left( P + \Delta_\ast \right) \geq r$ 
	is given by (\ref{eq:optimal_perturb}), but for a specific $\mu$ where the outer infimum 
	in (\ref{eq:main_sval_char}) is attained.
	\end{enumerate}
\end{theorem}

\subsection{Simplified Formula in terms of Divided Differences}\label{subsec:divided_difference}

The singular value characterization (\ref{eq:main_sval_char}) seems cumbersome at first.
It can be expressed in a much more comprehensible way by the use of the divided
differences, and the theorem below regarding the matrix functions of triangular matrices
\cite[Corollary of Theorem 2]{Davis1973}, \cite[Theorem 3]{VanLoan1975}, \cite[Theorem 4.11]{Higham2008}.
Recall that, for a function $f : {\mathbb R} \rightarrow {\mathbb R}$, we define the divided
difference at the nodes $x_0, \dots, x_k \in {\mathbb R}$ - where equal nodes are allowed
but must be contagious (i.e. $x_j = x_\ell$ for $\ell > j$ implies $x_i = x_j$ for all $i \in [j,\ell]$) -
recursively by the formula
\begin{equation}\label{eq:divided_differences}
	f \left[ x_0, x_1, \dots, x_k \right]
				=
	\begin{cases}
			\frac{ f \left[ x_1,\dots, x_k \right] - f\left[ x_0,\dots, x_{k-1} \right] }{ x_k - x_0 }	&	x_0 \neq x_k	\\
			\frac{f^{(k)}(x_0)}{k!}												&	x_0 = x_k
	\end{cases}
\end{equation}
and $f[x_j] = f(x_j)$ for each $j$ \cite{DeBoor2005}, \cite[Section 8.2.1]{Quarteroni2007}, \cite[Section B.16]{Higham2008}.
\begin{theorem}[\textbf{Functions of Triangular Matrices}]\label{thm:fun_tri_matrices}
Let $T$ be an $n\times n$ lower triangular matrix, and $f : {\mathbb R} \rightarrow {\mathbb R}$
be a function defined on the spectrum of $T$. Then ${\mathcal T} := f(T)$ is lower triangular
with ${\mathcal T}_{ii} = f(\mu_i)$ and
\[
	{\mathcal T}_{i\ell}
			=
	\sum_{(s_0, s_1, \dots, s_k)} t_{s_1 s_0} t_{s_2 s_1} \dots t_{s_k s_{k-1} } f\left[ \mu_{s_0}, \dots, \mu_{s_k} \right]
\]
for $i > \ell$, where $\mu_i = t_{ii}$, and the summation is over all increasing sequences of positive integers starting with $\ell$ 
and ending with $i$.
\end{theorem}

Now, letting $p_j(x) = x^j$, the formula in (\ref{eq:main_sval_char}) concerns the optimization
of the $r$th smallest singular value of
\[
	{\mathcal Q}(\mu,\Gamma,P)	  =	 \sum_{j=0}^m \; p_j \left(  C(\mu,\Gamma)^T \right)   \otimes A_j.
\]
Partition ${\mathcal Q}(\mu,\Gamma,P)$ into $n\times n$ blocks, then by an application of Theorem
\ref{thm:fun_tri_matrices} for $i > \ell$, its $n\times n$ submatrix at the $i$th block row and $\ell$th block
column is given by
\begin{eqnarray*}
	\displaystyle
		\sum_{j=0}^m 	\left( p_j \left(  C(\mu,\Gamma)^T \right) \right)_{i\ell} A_j
			& =
		\displaystyle
		\;\;
		\sum_{j=0}^m
		\sum_{(s_0, s_1, \dots, s_k)} \gamma_{s_1 s_0} \gamma_{s_2 s_1} \dots \gamma_{s_k s_{k-1} } p_j\left[ \mu_{s_0}, \dots, \mu_{s_k} \right]    A_j \hskip 4.5ex \\
			& =
		\displaystyle
		\sum_{(s_0, s_1, \dots, s_k)} \gamma_{s_1 s_0} \gamma_{s_2 s_1} \dots \gamma_{s_k s_{k-1} } \left( \sum_{j=0}^m p_j\left[ \mu_{s_0}, \dots, \mu_{s_k} \right]    A_j \right) \\
			& =
		\displaystyle
		\sum_{(s_0, s_1, \dots, s_k)} \gamma_{s_1 s_0} \gamma_{s_2 s_1} \dots \gamma_{s_k s_{k-1} } P \left[ \mu_{s_0}, \dots, \mu_{s_k} \right]    \hskip 12ex
\end{eqnarray*}
where we define $P \left[ \mu_{s_0}, \dots, \mu_{s_k} \right]$ by the divided difference formula (\ref{eq:divided_differences})
by replacing $f$ with the matrix polynomial $P$. On the other hand the $n\times n$ submatrix of 
${\mathcal Q}(\mu,\Gamma,P)$ at the $i$th block row and column is given by
\[
	\sum_{j=0}^m 	\left( p_j \left(  C(\mu,\Gamma)^T \right) \right)_{ii} A_j
				=
	\sum_{j=0}^m 	\mu_i^j A_j 
				=
				P(\mu_i).
\]

\begin{theorem}[\textbf{Divided Difference Characterization}]\label{thm:main_result2}
	Let $P(\lambda) := \sum_{j=0}^m \lambda^j A_j$ with $A_j \in {\mathbb C}^{n\times n}$ and such that
	${\rm rank}(A_m) = n$, ${\mathbb S}	:=	\{	\lambda_1, \dots, \lambda_s 	\}$ be a set of complex 
	scalars, and $r\in {\mathbb Z}^+$. Then the singular value characterization
	\begin{equation}\label{eq:main_sval_char2}
		\tau_r({\mathbb S})
				=
		\inf_{\mu \in {\mathbb S}^r} \sup_{\Gamma\in {\mathbb C}^{r(r-1)/2}} \;
				\sigma_{- r}
					\left(
					 	{\mathcal Q}(\mu,\Gamma,P)
					\right)
	\end{equation}
	holds, for the distance $\tau_r({\mathbb S})$ defined as in (\ref{eq:distance_defn}), provided that the
	multiplicity and linear independence assumptions hold at the optimal $\mu \in {\mathbb S}^r$
	and $\Gamma \in {\mathbb C}^{r(r-1)/2}$, and if $r > n$ provided that the inner supremum is attained,
	where ${\mathcal Q}(\mu,\Gamma,P) \in {\mathbb C}^{nr\times nr}$
	is block lower triangular whose $n\times n$ submatrix at rows $1+(i-1)n : in$ and at columns $1+(\ell-1)n : \ell n$ is 
	given by
	\[
		\begin{cases}
		\sum_{(s_0, s_1, \dots, s_k)} \gamma_{s_1 s_0} \gamma_{s_2 s_1} \dots \gamma_{s_k s_{k-1} } P \left[ \mu_{s_0}, \dots, \mu_{s_k} \right]  & i > \ell \\
		P(\mu_i)	&	i = \ell \\
		0		&	i < \ell
		\end{cases}
	\] 
	with summation over all positive increasing sequences starting with $\ell$ and ending with $i$.
\end{theorem}

The min-max characterization in (\ref{eq:main_sval_char2}) takes the form
\[
	\inf_{\mu_1, \mu_2 \in {\mathcal S}} \;\;
	\sup_{\gamma \in {\mathbb C}} \;
	\sigma_{-2}
		\left(
		\left[
			\begin{array}{cc}
				P(\mu_1)	  				&	0	\\
				\gamma P[\mu_1,\mu_2]		&	P(\mu_2)	\\
			\end{array}
		\right]
		\right)
\]
for the particular case $r = 2$ (i.e., two eigenvalues are prescribed), and
\[
	\inf_{\mu_1, \mu_2, \mu_3 \in {\mathcal S}} \;\;
	\sup_{\gamma_{21}, \gamma_{31}, \gamma_{32} \in {\mathbb C}} \;
	\sigma_{-3}
		\left(
		\left[
			\begin{array}{ccc}
				P(\mu_1)	  					&	0							&	0	\\
				\gamma_{21} P[\mu_1,\mu_2]		&	P(\mu_2)						&	0	\\
				\gamma_{21}\gamma_{32} P[\mu_1,\mu_2,\mu_3] 	+	\gamma_{31} P[\mu_1,\mu_3]		&	\gamma_{32} P[\mu_2,\mu_3]		&	P(\mu_3)	\\	
			\end{array}
		\right]
		\right)
\]
for $r = 3$ (i.e., three eigenvalues are prescribed). When $r = 2$, the inner maximization can be
performed over ${\mathbb R}$ rather than ${\mathbb C}$; observe that the singular values of the
matrix function remain the same if $\gamma$ is replaced by $|\gamma|$. Similarly, formulas for 
$r > 3$ can be obtained.

A particular case of interest is the distance to a nearest polynomial with an eigenvalue of algebraic
multiplicity $\geq r$. This distance was initially considered by Wilkinson \cite{Wilkinson1971, Wilkinson1984} 
and Ruhe \cite{Ruhe1970} for matrices due to its connection with the sensitivity of eigenvalues.
It has been extensively studied for matrices; see \cite{Malyshev1999} for
$r = 2$, \cite{Ikramov2003, Ikramov2005} for $r=3$ and \cite{Mengi2011} for an arbitrary $r$. For matrix polynomials,
a singular value characterization is derived in \cite{Papat2008} for $r=2$. For matrix polynomials and for an 
arbitrary $r$ we apply Theorem \ref{thm:main_result2} with ${\mathbb S} = \{ \mu \}$ leading us to following
characterization for the distance
\begin{equation}\label{eq:distance_mult}
	{\mathbb M}(\mu)
		:=
	\inf
			\{
				\| \Delta \|_2	\; | \;
					P(\lambda) + \Delta \;\; {\rm has} \; \mu \; {\rm as} \; {\rm an} \; {\rm eigenvalue}	 \;
					{\rm of} \; {\rm algebraic} \; {\rm multiplicity} \geq r		
			\}. 
\end{equation} 
\begin{corollary}[\textbf{Distance to Polynomials with Multiple Eigenvalues}]\label{thm:main_result3}
	Let $P(\lambda) := \sum_{j=0}^m \lambda^j A_j$ with $A_j \in {\mathbb C}^{n\times n}$ and such that
	${\rm rank}(A_m) = n$, $\mu \in {\mathbb C}$, and $r\in {\mathbb Z}^+$. Then the singular value characterization
	\begin{equation}\label{eq:main_sval_char3}
		{\mathbb M}(\mu)
				=
		 \sup_{\Gamma\in {\mathbb C}^{r(r-1)/2}} \;
				\sigma_{- r}
					\left(
					 	{\mathcal Q}(\mu,\Gamma,P)
					\right)
	\end{equation}
	holds, for the distance ${\mathbb M}(P,\mu)$ defined as in (\ref{eq:distance_mult}), provided that the
	multiplicity and linear independence assumptions hold at the optimal
	$\Gamma \in {\mathbb C}^{r(r-1)/2}$, and if $r > n$ provided that the supremum is attained,	
	where ${\mathcal Q}(\mu,\Gamma,P) \in {\mathbb C}^{nr\times nr}$ is block lower triangular whose 
	$n\times n$ submatrix at rows $1+(i-1)n : in$ and columns $1+(\ell-1)n : \ell n$ is given by
	\[
		\begin{cases}
		\sum_{(s_0, s_1, \dots, s_k)} \gamma_{s_1 s_0} \gamma_{s_2 s_1} \dots \gamma_{s_k s_{k-1} } \frac{P^{(k)}(\mu)}{k!}  & i > \ell \\
		P(\mu)	&	i = \ell \\
		0		&	i < \ell
		\end{cases}
	\] 
	with summation over all positive increasing sequences starting with $\ell$ and ending with $i$.
\end{corollary}
\noindent
Minimizing ${\mathbb M}(\mu)$ over all $\mu \in {\mathbb C}$ yields the distance to a nearest
polynomial with an eigenvalue of algebraic multiplicity $\geq r$.

Of particular instances of the formula (\ref{eq:main_sval_char3}) are
\[
	\sup_{\gamma \in {\mathbb C}} \;
	\sigma_{-2}
		\left(
		\left[
			\begin{array}{cc}
				P(\mu)	  				&	0	\\
				\gamma P'(\mu)		&	P(\mu)	\\
			\end{array}
		\right]
		\right),
\]
when $r = 2$, that is the distance to polynomials with $\mu$ as a multiple eigenvalue, which was also
derived in \cite{Papat2008}, and 
\[
	\sup_{\gamma_{21}, \gamma_{31}, \gamma_{32} \in {\mathbb C}} \;
	\sigma_{-3}
		\left(
		\left[
			\begin{array}{ccc}
				P(\mu)	  			&	0									&	0	\\
				\gamma_{21} P'(\mu) 	&	P(\mu)								&	0	\\
				\gamma_{21}\gamma_{32} P''(\mu)/2 	+	\gamma_{31} P'(\mu)		&	\gamma_{32} P'(\mu)		&	P(\mu)	\\	
			\end{array}
		\right]
		\right)
\]
when $r = 3$, that is the distance to polynomials with $\mu$ as a triple eigenvalue.

\section{ Numerical Examples }\label{sec:num_examp}

We illustrate our main results Theorem \ref{thm:main_result2} and Corollary \ref{thm:main_result3}
on two examples that can be visualized by means of the $\epsilon$-pseudospectrum of the
polynomial $P(\lambda)$. The $\epsilon$-pseudospectrum that is related to our results
consists of the eigenvalues of all polynomials within an $\epsilon$ neighborhood with respect
to the 2-norm and when only the constant perturbations are allowed, that is
\begin{eqnarray*}
	\Lambda_\epsilon (P)  &	:=	&	\bigcup_{\| \Delta \|_2 \leq \epsilon}	\Lambda \left( P + \Delta \right) \\
					    &	=	&	\{	z \in {\mathbb C}	\;\; | \;\;	\sigma_{-1} (P(z))	\leq \epsilon	\}.
\end{eqnarray*}
where $\Lambda(P)$ denotes the spectrum of the polynomial $P(\lambda)$.

The derivation in the previous section establishes that any stationary point of the inner
maximization problem in (\ref{eq:main_sval_char}) is a global maximizer as long as the 
multiplicity and linear independence assumptions hold. Consequently, we solve the inner 
problems using quasi-Newton methods numerically. For the numerical solutions of the outer 
minimization problems we depend on the technique recently described in \cite{Kilic2012}, 
which exploits the smoothness properties of a singular value function of a matrix function
depending on a parameter analytically.

Both of the numerical experiments below is performed on a $5\times 5$ matrix polynomial
of degree two, whose entries are selected from a normal distribution with zero mean
and unit variance.

\subsection{Polynomials with Two Prescribed Eigenvalues}
Suppose that ${\mathbb S} = \{ \lambda_1, \lambda_2 \}$ and $r=2$ so that two eigenvalues
are prescribed, and the distance to a nearest polynomial for which at least two of the eigenvalues 
belong to ${\mathbb S}$ is sought. Then the singular value formula (\ref{eq:main_sval_char2}) 
takes the form
\begin{equation}\label{eq:dist_twoeigs}
	\tau_r({\mathbb S})
				=
		\inf_{\mu \in {\mathbb S}^2} \; \sup_{\gamma} \;\;
				\sigma_{- 2} 
					\left(
						\left[
							\begin{array}{cc}
								P(\mu_1)	&	0 \\
								\gamma P [\mu_1,\mu_2]  &	P(\mu_2) \\	
							\end{array}
						\right]
					\right)
\end{equation}
where
\[
	 P [\mu_1,\mu_2]
	 			=
	 \begin{cases}
                  \left( \frac{P(\mu_1) - P(\mu_2)}{\mu_1 - \mu_2} \right) & \text{if $\mu_1 \neq \mu_2$}, \\
                  P'(\mu_1) & \text{if $\mu_1 = \mu_2 $}.
           \end{cases}
\]

Here, we calculate this distance for the quadratic matrix polynomial mentioned at
the beginning of this section with random entries, and for the prescribed eigenvalues
${\mathbb S} =\{ -0.3 + 0.1 i,  -0.65 \}$. The boundaries of the pseudospectra 
of the quadratic matrix polynomial are plotted in Figure \ref{fig:pres_two} together with
the prescribed eigenvalues marked by asterisks. In particular the outer curves correspond 
to the boundary of the $\epsilon$-pseudospectrum for $\epsilon = 0.5879$, which is the computed 
distance $\tau_r({\mathbb S})$ by means of the characterization (\ref{eq:dist_twoeigs}). On one of
these outer curves one of the prescribed eigenvalues $\lambda_1 = -0.3+0.1i$ lies. However,
in general it is possible that neither of the prescribed eigenvalues lies on the boundary
of the $\epsilon$-pseudospectrum for $\epsilon = \tau_r({\mathbb S})$;
both of the prescribed eigenvalues may possibly lie strictly inside the pseudospectrum.

\begin{figure}
   	\begin{center}
		\includegraphics[height=0.45\vsize,]{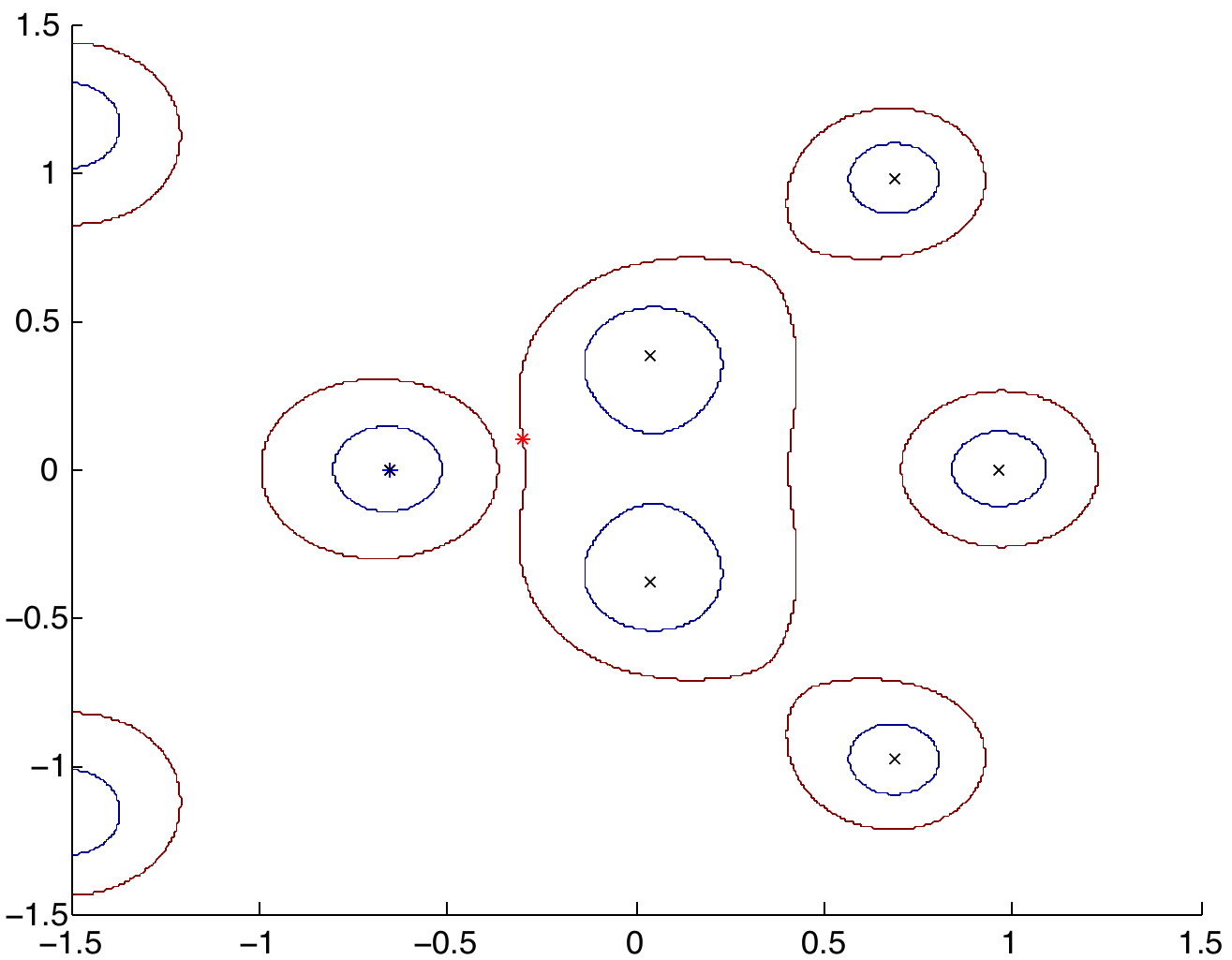}
	\end{center}
	    \caption{  The pseudospectra of the quadratic random matrix
	    polynomial are displayed. The asterisks are the prescribed eigenvalues. The black crosses
	    represent the eigenvalues of the matrix polynomial. One of the prescribed
	    eigenvalues is on the boundary of the $\epsilon$-pseudospectrum (outer curve)
	    for $\epsilon = \tau_r({\mathbb S})$.   
	    }\label{fig:pres_two}
\end{figure}

\subsection{Nearest Polynomials with Multiple Eigenvalues}
By Corollary \ref{thm:main_result3} the distance to a nearest matrix polynomial with 
a multiple eigenvalue is given by
\[
	\inf_{\mu \in {\mathbb C}} \;
	\sup_{ \gamma \in {\mathbb R} } \;\;
			\sigma_{-2}
			\left(
				\left[
					\begin{array}{cc}
						P(\mu)			&	0	\\
						\gamma P'(\mu)		&	P(\mu) \\
					\end{array}
				\right]
			\right).
\]
Indeed, it can be shown that this formula remains valid even when the multiplicity and linear
independence assumptions are violated. For a matrix polynomial of size $n\times n$ and 
degree $m$ the $\epsilon$-pseudospectrum for small $\epsilon$ is comprised of $nm$ 
disjoint components, one around each eigenvalue. The smallest $\epsilon$ such that two 
components of the $\epsilon$-pseudospectrum coalesce is equal to this distance. This is not 
an obvious fact; indeed for matrices this has been established by Alam and Bora \cite{Alam2005} 
not long time ago. The extensions for matrix pencils and matrix polynomials are given in
\cite[Theorem 5.1]{Ahmad2010}  and \cite[Theorem 7.1]{Ahmad2009}, respectively.

For the random quadratic matrix polynomial we compute this distance as 
$0.3211$. Two components of the $\epsilon$-pseudospectrum for $\epsilon = 0.3211$
coalesce as expected in theory. This is illustrated in Figure \ref{fig:multiple}; specifically
the inner-most curves represent the boundary of this $\epsilon$-pseudospectrum.
The point of coalescence of the components $z = 0.0490$, marked by an asterisk, is the 
multiple eigenvalue of a nearest polynomial.

\begin{figure}
   	\begin{center}
		\includegraphics[height=0.45\vsize,]{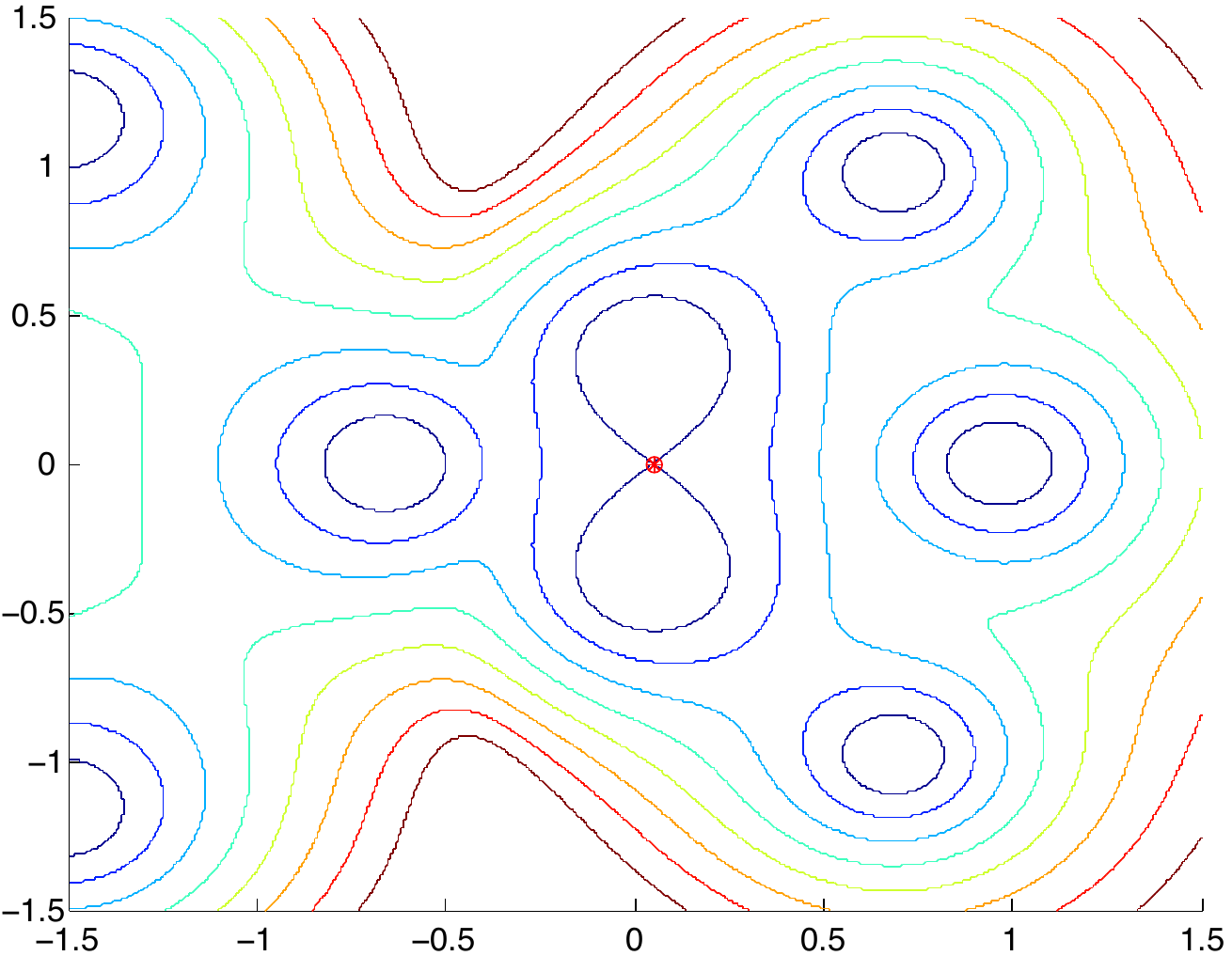}
	\end{center}
	    \caption{  The $\epsilon$-pseudospectra of the quadratic random matrix
	    polynomial for various $\epsilon$ are illustrated. The inner most curve corresponds to the
	    boundary of the $\epsilon$-pseudospectrum for $\epsilon$ equal to
	    the distance to a nearest matrix polynomial with a multiple eigenvalue. The red
	    asterisk is the multiple eigenvalue of a nearest matrix polynomial. }\label{fig:multiple}
\end{figure}

\section{ Concluding Remarks }\label{sec:conc}

We derived a singular value optimization characterization for the distance
from a matrix polynomial to a nearest one with a specified number of eigenvalues
belonging to a specified set. We restricted ourselves to square matrix
polynomials. Extensions to rectangular matrix polynomials are straightforward 
as long as the leading coefficient matrix is full rank.

There are two important open problems that are left untouched by this paper.
First, it is more desirable to allow perturbations to all coefficient matrices from an 
application point of view. In this case an exact singular value formula is not known 
at the moment. Secondly, the results are proven under mild multiplicity and linear
independence assumptions. Our experience with special instances indicates that 
the singular value formula remains valid, even when these assumptions are not 
met. 

\vskip 4ex

\noindent
\textbf{Acknowledgement:} We are grateful to Daniel Kressner for reading an initial version,
and helpful suggestions, which eventually led us to the simplified divided difference formulas 
in Section \ref{subsec:divided_difference} for the singular value characterization (\ref{eq:main_sval_char}).

\appendix

\section{Proof of attainment of the supremum of the singular value function}

Below we establish that, for generic values of $\mu \in {\mathbb C}^r$, the supremum of
\[
	\sigma_{-r}
		\left(
			{\mathcal Q}(\mu,\Gamma,P)
		\right)
\]
over all $\Gamma \in {\mathbb C}^{(r-1)r/2}$ is attained for $r \leq n$. Here the block lower triangular matrix
${\mathcal Q}(\mu,\Gamma,P) \in {\mathbb C}^{nr\times nr}$ is as in Theorem \ref{thm:main_result2}.
The attainment result here is a generalization of the result presented in the appendix in  \cite{Kressner2011}, and
its proof below mimics the proof over there.

\begin{theorem}\label{thm:decay_polynomial}
Suppose that $P\left[ \mu_{k}, \mu_{l}  \right]$ has full rank for each $k$ and $l$ such that
$k < l$. Then for $j = 1,\dots,n$ we have
\[
	\sigma_{-j}
		\left(
			{\mathcal Q}(\mu,\Gamma,P)
		\right)	
			\rightarrow	0
			\;\;\; {\rm as}	\;\;\;
			\Gamma \rightarrow \underline{\Gamma}	
\]
where $\| \underline{\Gamma} \| = \infty$.
\end{theorem}
\begin{proof}
Since $\| \Gamma \| \rightarrow \infty$, there exists a $\gamma_{lk}$ such that $| \gamma_{lk} | \rightarrow \infty$.
Choose an unbounded $\gamma_{lk}$ so that $l - k$ is as small as possible. Thus $| \gamma_{ij}|$ is bounded 
for each $i, j$ such that $i - j < l - k$. 

Let us first suppose that none of $\mu_1, \dots, \mu_r$ is an eigenvalue of $P(\lambda)$.
Our approach is based on establishing that the largest $n$ singular values of ${\mathcal Q}(\mu,\Gamma,P)^{-1}$
diverges to $\infty$ as $| \gamma_{lk} | \rightarrow \infty$. Clearly, this is equivalent to the decay of
the least $n$ singular values of ${\mathcal Q}(\mu,\Gamma,P)$ to zero. In this respect we claim 
that ${\mathcal Q}(\mu,\Gamma,P)^{-1}$ is of the form
\[
	\left[
			\begin{array}{cccccc}
				P(\mu_1)^{-1}	&	0				&	0			&	\dots		&				&	0				\\
				X_{21}		&	P(\mu_2)^{-1}		&	0			&			&				&	0				\\
				X_{31}		&	X_{32}			&     P(\mu_3)^{-1}	&			&									\\
							&					&				& 	\ddots	&									\\
							&					&				&			& P(\mu_{r-1})^{-1}	&	0				\\	
				X_{r1}		&	X_{r2}			&				&			& X_{r(r-1)}		&	P(\mu_r)^{-1}		\\
			\end{array}
	\right]	
\]
where 
\begin{equation}\label{eq:blocks}
	X_{lk} = -\gamma_{lk} P(\mu_l)^{-1} P[\mu_k, \mu_l] P(\mu_k)^{-1} + P_{\Delta,lk}
\end{equation}
and $P_{\Delta,lk}$ is a polynomial in $\gamma_{(k+1)k}, \dots, \gamma_{(l-1)k},$ $\dots, \gamma_{l(l-1)}$, which are all 
bounded. The proof of this later claim is by induction on $l-k$. As the base case, when $l = k+1$, we have 
\[
	X_{(k+1)k} = -\gamma_{(k+1)k} P(\mu_{k+1})^{-1} P[\mu_k, \mu_{k+1}] P(\mu_k)^{-1}.
\]
For the inductive case, let us partition ${\mathcal Q}(\mu,\Gamma,P)$ into $n\times n$ blocks and denote the
submatrix at the $i$th block row and $j$th block columns with ${\mathcal Q}_{ij}$. Then, by multiplying the $l$th
block row of ${\mathcal Q}(\mu,\Gamma,P)$ with the $k$th block column of its inverse for $l > k$ and letting
$X_{kk} = P(\mu_k)^{-1}$, we have $\sum_{j = k}^l {\mathcal Q}_{lj} X_{jk}	=	0$ implying
\begin{eqnarray*}
	X_{lk}	=	-\gamma_{lk} P(\mu_l)^{-1} P[\mu_k, \mu_l] P(\mu_k)^{-1}
						-
	\sum_{(s_0,s_1,\dots,s_j)} \gamma_{s_1 s_0} \dots \gamma_{s_j s_{j-1}} P(\mu_l)^{-1} P[\mu_{s_0}, \dots,\mu_{s_j}] P(\mu_k)^{-1} \\
					-
		\sum_{j = k+1}^{l-1}  P(\mu_l)^{-1} {\mathcal Q}_{lj} X_{jk}. \hskip 50ex
\end{eqnarray*}
where again the first summation is over all increasing sequences of integers of length at least two 
starting with $k$ and ending with $l$. By the inductive hypothesis $X_{jk}$ for $j = k+1,\dots,l-1$ is a 
polynomial in $\gamma_{(k+1)k}, \dots, \gamma_{(l-1)k},$ $\dots, \gamma_{l(l-1)}$ only. 
This confirms $(\ref{eq:blocks})$.

Now, due to the assumption that $P[\mu_k, \mu_l]$ is full rank, from (\ref{eq:blocks}) we have
$\sigma_j(X_{lk}) \rightarrow \infty$ for $j = 1,\dots,n$. Thus the inequality
\[
	\sigma_j(X_{lk}) 
		\leq
	\sigma_j \left(  {\mathcal Q}(\mu,\Gamma,P)^{-1}  \right)
\]
yields $\sigma_j \left(  {\mathcal Q}(\mu,\Gamma,P)^{-1}  \right) \rightarrow \infty$ for each $j = 1,\dots,n$
as desired.

Finally, if some $\mu_j$ are eigenvalues of $P(\lambda)$, for each $\beta > 0$ there exists a $\Delta \in {\mathbb C}^{n\times n}$
such that $\| \Delta \|_2 \leq \beta$ and $P_\Delta(\lambda) := P(\lambda) + \Delta$ does not have any of $\mu_j$ 
as  eigenvalues. The previous argument applies to $P_{\Delta}(\lambda)$, in particular the least $n$ singular
values of the associated Kronecker matrix decay to zero as $| \gamma_{lk} | \rightarrow \infty$. Thus, for some 
$\delta_{\beta}$ for all $\gamma_{lk}$ such that $|\gamma_{lk}| > \delta_{\beta}$ we have
\[
	 \sigma_{-j} \left(  {\mathcal Q} (\mu,\Gamma, P_\Delta ) \right) \; < \; \beta
			\;\;\;		\Longrightarrow	\;\;\;
	\sigma_{-j} \left(  {\mathcal Q} (\mu,\Gamma,P)  \right) < 2\beta
\]
completing the proof.
\end{proof}
\noindent
The previous theorem and the continuity of the singular values ensure that the
supremum of 
$
	\sigma_{-r}
		\left(
			{\mathcal Q}(\mu,\Gamma,P)
		\right)
$
over all  $\Gamma \in {\mathbb C}^{(r-1)r/2}$ is attained provided $r \leq n$.

The hypotheses that $P[\mu_k,\mu_l]$ are full rank hold generically over all pairs $(\mu_k, \mu_l)$.
If the degree of the polynomial $P(\lambda)$ is one, $P[\mu_k,\mu_l] = A_1$ is full rank for all $(\mu_k, \mu_l)$.
Otherwise, suppose $P[\mu_k,\mu_l] = \sum_{j=1}^m p_j[\mu_k,\mu_l] \cdot A_j$ is singular with $p_j(x) = x^j$. 
Since the leading coefficient $A_m$ is non-singular  and $p_j[\mu_k + \delta,\mu_l] - p_j[\mu_k,\mu_l]$
is a monic polynomial of $\delta$ of degree $j-1$, thus for $j > 1$ nonzero essentially everywhere,
$P[\mu_k+\delta,\mu_l] = \sum_{j=1}^m p_j[\mu_k+\delta,\mu_l] \cdot A_j$ is non-singular essentially 
for all small $\delta$.

\bibliography{polynomial_specified}
\end{document}